\documentclass[11pt]{article}
\usepackage{amsmath,amsfonts,amssymb,amsthm,paralist,xspace,amscd,array}
\usepackage{enumitem}
\usepackage[affil-it]{authblk}
\usepackage{a4wide}
\usepackage{color}
\usepackage[colorlinks=true, pdfstartview=FitV, linkcolor=blue,
citecolor=blue,
urlcolor=blue,
breaklinks=true]{hyperref}
\usepackage[usenames,dvipsnames,svgnames,table]{xcolor}
\usepackage{xfrac}
\usepackage{cancel}
\usepackage[normalem]{ulem}
\usepackage{bbm}
\usepackage{lmodern}

\newtheorem{thm}[equation]{Theorem}
\newtheorem{cor}[equation]{Corollary}
\newtheorem{prop}[equation]{Proposition}
\newtheorem{lemma}[equation]{Lemma}
\theoremstyle{definition}
\newtheorem{defn}[equation]{Definition}
\newtheorem{remark}[equation]{Remark}
\newtheorem{exam}[equation]{Example}
\numberwithin{equation}{section}

\setlist[enumerate]{itemsep=2pt, topsep=4pt}

\newcommand{\KK}{\mathbbm{k}}
\newcommand{\ZZ}{\mathbb{Z}}  
  
\newcommand{\N}{\mathbb{N}}

\newcommand{\Z}{\mathsf{Z}}

\newcommand{\q}{\quad}

\newcommand\chara{\mathsf{char}}

\newcommand{\pb}[1]{\left\{ #1\right\}}
\newcommand{\seq}[1]{\left( #1\right)}

\renewcommand{\gcd}{\mathsf{gcd}}

\newcommand\degg{\, \mathsf{deg} \,}
\newcommand\spann{\, \mathsf{span}}
\newcommand\Roots{\, \mathsf{R}}
\newcommand\Poly{\, \mathsf{poly}}

\newcommand{\Sg}[1][g]{\mathcal{S}(#1)}
\newcommand{\Su}[1][u]{\mathcal{S}_{#1}}
\newcommand{\md}{\ensuremath{\mathsf{-mod}}}

\usepackage{mathabx}
\usepackage{scalerel}
\newcommand{\nback}[1][-.95pt]{
  \mathrel{\raisebox{#1}{$\rotatebox[origin=c]{-315}{\scaleobj{0.55}{-}}$}}
}
\newcommand{\undernegpreccurlyeq}{%
\mathrel{\ooalign{$\preccurlyeq$\cr\kern1.2pt$\nback$}}}

\newcommand{\E}{\mathcal{E}}

\newcommand{\function}[3]{#1\colon #2 \rightarrow #3}

\newcounter{marg}[section]

\usepackage[normalem]{ulem}
\newcommand\rsout{\bgroup\markoverwith{\textcolor{red}{\rule[0.5ex]{2pt}{0.4pt}}}\ULon}

\title{Representations of Smith algebras which are free over the Cartan subalgebra}

\author[1,2]{Vyacheslav Futorny}

\affil[1]{Shenzhen International Center for Mathematics, Southern University of Science and Technology, Shenzhen, China.}

\affil[2]{Instituto de Matem\'atica e Estat\'istica, Universidade de S\~ao Paulo, S\~ao Paulo, Brasil.}

\author[3]{Samuel A.\ Lopes\thanks{Partially supported by CMUP, member of LASI, which is financed by national funds through FCT -- Funda\c c\~ao para a Ci\^encia e a Tecnologia, I.P., under the projects with reference UIDB/00144/2020 and UIDP/00144/2020.}}

\affil[3]{CMUP, Departamento de Matem\'atica, Faculdade de Ci\^encias, Universidade do Porto, Rua do Campo Alegre s/n, 4169--007 Porto, Portugal.}

\author[2]{Eduardo M.\ Mendonça\thanks{Supported by São Paulo Research Foundation (FAPESP), grants 2020/14313-4 and  2022/05915-6.}}

\date{{\textit{To the memory of Georgia Benkart }}}

\allowdisplaybreaks

\begin{document}

\maketitle


\begin{abstract}
In this paper, we study the category of modules over the Smith algebra which are free of finite rank over the unital polynomial subalgebra generated by the Cartan element $h$ and obtain families of such simple modules of arbitrary rank. In the case of rank one we obtain a full description of the isomorphism classes, a simplicity criterion, and an algorithm to produce all composition series. We show that all 
 such modules have finite length and describe the composition factors and their multiplicity. 
\newline\newline
\textbf{MSC Numbers (2020)}: Primary 16S30, 16S99
 \hfill \newline
\textbf{Keywords}: Smith algebra, simple module, $U(h)$-free module
\end{abstract}

\section{Introduction}

In~\cite{spS90}, Smith defined a class of algebras similar to the enveloping algebra of $\mathfrak{sl}_2$, essentially by replacing the standard relation $[e,f]=h$ in $U(\mathfrak{sl}_2)$ with the relation $[e,f]=g(h)$, where $g$ is an arbitrary polynomial in $h$. We will denote these algebras by $\Sg$. Among other results, Smith classified the finite-dimensional simple $\Sg$-modules, seen as quotients of Verma modules, and introduced an analog of the Bernstein--Gelfand--Gelfand category $\mathcal O$ for $\Sg$. 

The Smith algebras have been extensively studied and are related to down-up algebras, a class of algebras introduced by Benkart and Roby in \cite{BR98}, inspired by the relations satisfied by the down and up operators on a differential poset. Down-up algebras also display many similarities with enveloping algebras of three-dimensional Lie algebras, and include those Smith algebras $\Sg$ with $\degg g\leq 1$. Later, in \cite{CS04}, Cassidy and Shelton introduced generalized down-up algebras, which include all the Smith algebras $\Sg$. 

As with the enveloping algebra of $\mathfrak{sl}_2$, every Smith algebra has a Casimir element, which generates its center and acts as a scalar on simple $\Sg$-modules. The corresponding factor rings of $\Sg$ by the maximal ideal of the center have been considered by Joseph \cite[Lemma~3.1]{aJ77}, where simplicity criteria were given, and by Hodges \cite{tH93}, as algebras of invariants of the Weyl algebra under the action of a cyclic group. Allowing $\Sg$ to be defined over a ring, then some of these quotients can further be seen as invariant rings of differential operators on a multiplicity-free representation of an algebraic group under the action of its derived subgroup \cite{hR14}. Another interesting connection is with the Zhu algebra of a vertex operator algebra associated to a positive definite rank-one lattice, which is shown in \cite{DLM97} to be isomorphic to a finite-dimensional quotient of $\Sg$. 

Our main interest is the representation theory of the Smith algebras $\Sg$. As we mentioned, the finite-dimensional irreducible representations, the Verma modules and category $\mathcal O$ have already been investigated in~\cite{spS90} (see also \cite{JV95}). In Block's classification \cite{rB81} of simple $U(\mathfrak{sl}_2)$-modules, along with the weight modules one finds also Whittaker modules and other modules defined via localization, the latter being torsion free over the polynomial algebra in $h$. A class of modules which has recently gained a lot of attention in the context of Lie algebras is given by the modules which are free of finite rank over the enveloping algebra of a Cartan subalgebra. These have been introduced and studied in \cite{jN15, jN16, TZ18}, and they are in a certain sense opposite to weight modules, as the action of the Cartan subalgebra is torsion free, rather than semisimple.  In particular,  free rank one simple $\mathfrak{sl}_{n+1}$-modules were classified in \cite{jN15}, and in \cite{TZ18} such modules over $\mathfrak{sl}_{n+1}$ were also constructed from modules over Witt algebras $W_n$.
Similarly, such simple $\mathfrak{sp}_{2n}$-modules were classified in \cite{jN16}. These are the only simple finite-dimensional algebras for which
there exist modules that are free  over the enveloping algebra of a Cartan subalgebra. Parabolic induction from simple $U(h)$-free modules was studied in \cite{CLNZ18}.

In this paper, we investigate the category of $\Sg$-modules which are free of finite rank over the unital subalgebra generated by $h$ and obtain families of such simple modules of arbitrary rank. We dedicate particular attention to the case of rank one, where we obtain a full description of the isomorphism classes, a simplicity criterion, and an algorithm to produce all composition series, resulting in a proof that such modules have finite length and in a full description of the composition factors and their multiplicity.

\paragraph{Notations and conventions.}\hfill

We work over an algebraically closed field $\KK$ of characteristic zero. Since a monic polynomial in $\KK[h]$ is fully determined by its set of roots, we 
get a bijection between the finite submultisets of $\KK$ and monic polynomials in $\KK[h]$. For convenience, we associate
 the field $\KK$ to the zero polynomial. 
  Given any $f(h)\in\KK[h]$, we let $\Roots_f$ denote its multiset of roots. 
Conversely, given any finite multiset $X$ of elements of $\KK$, we let $\Poly_X \in \KK[h]$ denote the unique monic polynomial such that $\Roots_{\Poly_X} = X$, that is, $\Poly_X = \prod_{\lambda \in X} (h-\lambda)$. Adopting the usual convention that an empty product equals $1$, we assume $\Poly_\emptyset = 1$. Moreover, we follow the convention that $\degg 0=-\infty$, with its usual arithmetic properties.

Given a multiset $X$ of elements of $\KK$ and $\lambda\in\KK$, we denote by $X\setminus\pb{\lambda}$ (respectively, $X\cup\pb{\lambda}$) the multiset obtained from $X$ by reducing (respectively, increasing) by one the multiplicity of $\lambda$ in $X$, and proceed similarly for the difference and union of arbitrary multisets. For example, $\pb{1, 2, 2, 5, 5, 5}\setminus\pb{3,5}=\pb{1, 2, 2, 5, 5}$ and $\pb{1, 2, 2, 5, 5, 5}\cup\pb{3,5}=\pb{1, 2, 2, 3, 5, 5, 5, 5}$. The cardinality $|X|$ of the (finite) multiset $X$ is the sum of the multiplicities of its elements. The underlying set obtained from $X$ (by eliminating repeated elements) will be denoted by $\underline{X}$. Thus, $|\pb{1, 2, 2, 5, 5, 5}|=6$ and $\underline{\pb{1, 2, 2, 5, 5, 5}}=\pb{1, 2, 5}$.

For $n\in \N=\ZZ_{\geq 0}$, set $[n]=\pb{1,\ldots, n}$, so in particular $[0]=\emptyset$.

\paragraph{Acknowledgment.}\hfill

Part of this research was carried out during visits of the first and third named authors to the Faculty of Sciences (FCUP) and the Center of Mathematics of the University of Porto (CMUP). They would like to express their gratitude for the hospitality received. Moreover, the third named author would like to thank Professor Olivier Mathieu for helpful conversations concerning the subject of this article.

\section{The Smith algebra}
Fix a polynomial $g(h) \in \KK[h]$. The \emph{Smith algebra} $\Sg$ is the unital associative algebra over $\KK$ generated by $x,y,h$ with definition relations:
\begin{equation}\label{eq:rel-gheis}
    [h,y] = y, \quad [h,x] = -x \quad \text{and}\quad [y,x]=g(h).
\end{equation}
This algebra was introduced by Smith in \cite{spS90}. In case $g(h)=0$, the generators $x$ and $y$ commute and the representation theory of $\Sg[0]$ assumes characteristics which often diverge from the general theory in case $g\neq 0$. In fact, $\Sg[0]$ is the enveloping algebra of a 3-dimensional solvable (non-nilpotent) Lie algebra. \textbf{Thus, henceforth we will always implicitly assume that $g\neq 0$.}

By \cite[Lemma~1.4]{spS90}, there exists $u(h)\in \KK[h]$ such that $g(h) = u(h-1) - u(h)$. Moreover, $u(h)$ is uniquely determined up to its constant term, which can be arbitrary, and $\degg (u)=\degg (g) +1\geq 1$. Fixing one such $u$, we denote the Smith algebra $\Sg$ by $\Su$ and remark that $\Su=\Su[u+C]$, for any $C\in\KK$.

Let $z_u = xy-u(h)=yx-u(h-1)$. It is easy to see that $z_u$ is a central element in $\Su$ and for this reason we call $z_u$ the \emph{Casimir element} associated with $u$. In addition, the center $\Z(\Su)$ of $\Su$ is $\KK[z_u]$, the polynomial algebra in $z_u$ (see \cite[Proposition~1.5]{spS90}, or \cite[Proposition~2.9]{LR22jaa}).

Next we show that the algebra $\Su$ acts on the polynomial algebra $\KK[t]$ by differential operators, with scalar central character. Denote by $A_1 = \KK[t,\partial]$, where $\partial=\frac{d\ }{dt}$, the first Weyl algebra over $\KK$, realized here as the ring of differential operators on $\KK[t]$ with polynomial coefficients. Since $g\neq 0$, then $\degg (u)=\degg (g) +1\geq 1$ and hence $\Roots_{u + C} \ne \emptyset$ for any $C\in\KK$. For $C\in\KK$, every root $\lambda \in \Roots_{u+C}$ and submultiset $X \subseteq \Roots_{u+C}\setminus\{\lambda\}$ define the polynomials:
\[
    Q_X(h) = \Poly_X \quad \text{and}\quad P_X(h) = \dfrac{u(h-1)+C}{Q_X(h-1)(h-(\lambda+1))}.
\]
Equivalently, $P_X(h) = \xi \Poly_{\Roots_{u+C}\setminus(\{\lambda\}\cup X)}(h-1)$, where $\xi$ is the leading coefficient of $u+C$.
\begin{lemma}\label{lem:Su-in-A1}
    There exists a morphism of algebra $\function{\varphi_{C,\lambda,X}}{S_u}{A_1}$ such that
    \[
        x \mapsto \partial Q_X(\theta - 1),\quad y \mapsto t P_X(\theta + 1) \quad \text{and}\quad h \mapsto \theta,
    \]
    where $\theta = t\partial + \lambda + 1$. Moreover, $z_u$ is mapped to $C$.
\end{lemma}
\begin{proof}
Define actions of $x,\,y$ and $h$ on $\KK[t]$ as in the statement above. 
Since $\theta t= t(\theta+1)$ and $\theta \partial=\partial(\theta-1)$, it follows that
\begin{align*}
    \seq{t  P_X(\theta + 1)}\seq{\partial  Q_X(\theta-1)}
        &= t \partial P_X(\theta) Q_X(\theta - 1) = ((\theta - 1) - \lambda) P_X(\theta)  Q_X(\theta-1), \\
    \seq{\partial  Q_X(\theta -1 )}\seq{t  P_X(\theta+1)}
        &=\partial  t Q_X(\theta)  P_X(\theta+1)= (\theta-\lambda)  P_X(\theta+1)Q_X(\theta).
\end{align*}
Then, from the equality $u(h) + C = (h - \lambda) P_X(h+1) Q_X(h)$, it follows that $[y,x]$ acts on $\KK[t]$ as $((\theta - 1) - \lambda) P_X(\theta)  Q_X(\theta-1) - (\theta-\lambda)  P_X(\theta+1)Q_X(\theta) = u(\theta - 1) + C - u(\theta) - C = g(\theta)$, which is the action of $g(h)$. Similarly, the relations $[h,y] = y$ and $[h,x] = -x$ are also preserved by the action, thus inducing an $\Su$-module structure by differential operators on $\KK[t]$, and hence the given morphism of algebras. It is straightforward to show that $\varphi_{C,\lambda,X}(z_u) = C$.
\end{proof}

\begin{exam}
Let $g(h)=h$, so that $\Sg[h]\simeq U(\mathfrak{sl}_2)$, the universal enveloping algebra of $\mathfrak{sl}_2$. Then we can take $u(h)=-\frac{1}{2} h(h+1)$, $C=0$, $\lambda=0$ and $X=\pb{-1}$, so that $Q_X(h)=h+1$ and $P_X(h)=-\frac{1}{2}$. We obtain an action of $\mathfrak{sl}_2$ on $\KK[t]$ where $x$ acts by $\partial (t\partial+1)$, $y$ acts by $-\frac{1}{2}t$ and $h$ acts by $t\partial+1$. 

Concretely,
\begin{align*}
x\cdot t^k=k(k+1)t^{k-1},\quad  y\cdot t^k=-\frac{1}{2}t^{k+1}\quad\text{and}\quad h\cdot t^k=(k+1)t^{k},\quad\text{for all $k\geq0$.}
\end{align*}
Using the fact that the action of $x$ lowers the degree
in $t$, annihilating only the constant polynomials, and the action of $y$ raises it, a straightforward argument shows that this is an irreducible representation of $\mathfrak{sl}_2$.
\end{exam}

As a consequence of the previous lemma, and using exponential modules for the Weyl algebra (compare \cite{GN22} for the case of $\mathfrak{sl}_2$),
we can construct a class of non-weight representations of $\Su$ as follows.
\begin{defn}\label{D:exp-mods}(Exponential modules)
Let $p \in \KK[t]$ be a polynomial and consider the $A_1$-module $\KK[t]e^p$. Given $C \in \KK$, $\lambda \in \Roots_{u+C}$ and $X \subseteq \Roots_{u+C}\setminus\{\lambda\}$ a submultiset, define $\E(p,C,\lambda,X)$ to be the $\Su$-module induced from the $A_1$-module $\KK[t]e^p$ via the map $\varphi_{C,\lambda,X}$ from Lemma~\ref{lem:Su-in-A1}.
\end{defn}

\begin{thm}\label{lem:class-h-free-mod}
Assume that $\degg p\geq 1$. Then $\E(p,C,\lambda,X)$ is a $\KK[h]$-free module of rank $\degg p$. Furthermore, if 
there is no $\mu \in \Roots_{u+C}\setminus\{\lambda\}$ such that $\mu - \lambda \in \ZZ_{\ge 1}$ then $\E(p,C,\lambda,\Roots_{u+C}\setminus\{\lambda\})$ is simple.
\end{thm}
\begin{proof}
    Let $n = \degg p$. We claim that $B= \{e^p,te^p,\dotsc,t^{n-1}e^p\}$ is a $\KK[h]$-basis of $\E(p,C,\lambda,X)$.
    
    From the relation $(h-(\lambda + 1 + s)) \cdot t^s e^p = t^{s+1}p'e^p$ we can show, by induction on $s$, that $t^s e^p\in\KK[h]B$ for all $s\in\N$, so we conclude that $B$ generates $\KK[t]e^p$ as a $\KK[h]$-module. Now notice that $h \cdot qe^p = (t(q' + qp') + (\lambda +1)q)e^p$, for all $q \in \KK[t]$. In particular, $h \cdot qe^p = \hat q e^p$, where $\degg \hat q = \degg p + \degg q$. Thus, we can conclude that $r(h)\cdot qe^p = \hat q e^p$, for some $\hat q\in\KK[h]$ such that
    \[
        \degg \hat q = (\degg r)(\degg p) + \degg q.
    \]
Suppose, by contradiction, that $\sum_{i = 0}^{n-1} r_i(h)\cdot t^i e^p=0$, for some $r_i \in \KK[h]$, not all zero. If there is a unique $i$ such that $r_i\neq 0$, then $0=r_i(h)\cdot t^i e^p=\hat q e^p$ with $\degg \hat q =i+ n\degg r_i\geq 0$. Thus $\hat q \neq 0$, which contradicts the equality $\hat q e^p=0$. So assume that at least two of the $r_i$ are nonzero. Then there are $0 \le i < j \le n-1$ such that $r_i, r_j\neq 0$ and $(\degg r_i)n + i = (\degg r_j)n + j$. Hence $(\degg r_i - \degg r_j)n = j-i \in [n-1]$. As $[n-1]$ contains no multiples of $n$, this is impossible. Therefore $B$ is $\KK[h]$-linearly independent and the claim is proved. 

        Now consider the case $X = \Roots_{u+C}\setminus\{\lambda\}$. In this case, $y$ acts as multiplication by $\beta t$, for some $\beta \in \KK^\times$. Replacing $y$ with $y/\beta$, we can, and will, assume that $\beta=1$, for simplicity, so that $y$ acts as multiplication by $t$.
        
        Let $V \subseteq \E(p,C,\lambda,X)$ be a nonzero submodule. We claim that $t^ie^p \in V$, for some $i \in \N$. First, notice that $(h - (\lambda + 1) - yp'(y))qe^p = tq' e^p$, for all $q \in \KK[t]$. In particular, $(h - (\lambda + 1) - yp'(y))t^j e^p = jt^j e^p$, for all $j \in \N$, so $\pb{t^je^p\mid j\in\N}$ is a basis of $\E(p,C,\lambda,X)$ of eigenvectors for the action of $(h - (\lambda + 1) - yp'(y))$. Thus, this operator has a diagonal action on $\E(p,C,\lambda,X)$ and hence also on $V$. Since the eigenspaces are one-dimensional, $V$ must contain some eigenvector, say $t^ie^p$, for some $i \in \N$.

Let $i \in \N$ be minimum such that $t^ie^p\in V$. Using induction on the number of elements of $X$, one can prove that $$Q_X(h) t^j e^p = \left(\prod_{\mu \in X}(\lambda + 1 + j - \mu) + tq_j \right)t^j e^p,$$ for $j\in\N$ and for some $q_j \in \KK[t]$. Since 
        \begin{align*}
          V \ni  x \cdot t^i e^p 
                &= Q_X(\theta) \partial t^i e^p = Q_X(h)(it^{i-1} + t^ip')e^p \\
                &= i \left(\prod_{\mu \in X}(\lambda + i - \mu) + tq_{i-1} \right)t^{i-1} e^p + Q_X(h)p'(y)t^ie^p\\ 
                &= i \prod_{\mu \in X}(\lambda + i - \mu)t^{i-1}e^p + \left(iq_{i-1}(y) + Q_X(h)p'(y)\right)t^i e^p,
        \end{align*}
we deduce that $i\prod_{\mu \in X}(\lambda + i - \mu) t^{i-1}e^p\in V$. By the minimality of $i$, we conclude that $i\prod_{\mu \in X}(\lambda + i - \mu)=0$ and from the hypothesis that $\mu - \lambda \notin \ZZ_{\ge 1}$ for all $\mu \in X$, it must be that $i=0$. Therefore $V=\E(p,C,\lambda,X)$ and the simplicity of $\E(p,C,\lambda,\Roots_{u+C}\setminus\{\lambda\})$ is established.
\end{proof}

\section{The category $\mathfrak{U}$ of $\KK[h]$-free $\Su$-modules}

Denote by $\mathfrak{U}$  the category of $\Su$-modules that are free of finite rank over the subalgebra $\KK[h]$. In this section we describe a skeleton of the category $\mathfrak{U}_1$, the full subcategory of $\mathfrak{U}$ consisting of modules that are free of rank one over $\KK[h]$. We show that any module in $\mathfrak{U}_1$ is of finite length, give an algorithm to determine all of its composition series, and give an explicit classification of the simple objects in $\mathfrak{U}_1$.

Let $M \in \mathfrak{U}$ have rank $n$, so we may assume that $M = \KK[h]^n$ as a $\KK[h]$-module. Let $1_1,\dotsc,1_n \in\KK[h]^n$ be its canonical basis. We have
\[
    y(h^k\cdot1_i) = (h-1)^ky1_i\quad \text{and}\quad x(h^k \cdot 1_i) = (h+1)^k x \cdot 1_i, \quad \text{for }i \in [n] \text{ and } k\in\N.
\]
Therefore,
\begin{equation}\label{eq:comutation-with-x-and-y}
    y f(h)\cdot1_i = f(h-1) y\cdot 1_i\quad \text{and}\quad x f(h)\cdot1_i = f(h+1) x\cdot 1_i,  \quad \text{for }i \in [n] \text{ and } f(h)\in\KK[h].
\end{equation}
In particular, the action of $\Su$ on $M$ is uniquely defined by a choice
\begin{align}\label{def:p,q-1}
    y\cdot1_i &=\colon p_i = (p_{i,1},p_{i,2},\dotsc,p_{i,n}) \in \KK[h]^n,\\
    x\cdot1_i &=\colon q_i = (q_{i,1},q_{i,2},\dotsc,q_{i,n}) \in \KK[h]^n,\label{def:p,q-2}
\end{align}
for all $i \in [n]$. By considering that $[y,x]=g(h)=u(h-1)-u(h)$, we deduce that the $p_{i,j}$ and the $q_{i,j}$ must satisfy the relations
\[
    g(h)1_i =  \sum_{\ell = 1}^n\left(\sum_{j=1}^{n} q_{i,j}(h-1)p_{j,\ell}(h) - p_{i,j}(h+1)q_{j,\ell}(h)\right)1_\ell,
\]
for all $i\in [n]$. Writing $Q = (q_{i,j}), P = (p_{i,j}) \in M_n(\KK[h])$, we see that the above is equivalent to the following matrix equation over $\KK[h]$:
\begin{equation}\label{eq:condition-p,q}
    Q(h-1)P(h) - P(h+1)Q(h) = g(h)I,
\end{equation}
where $I \in M_n(\KK[h])$ is the identity matrix. In fact, it is easy to see that \eqref{def:p,q-1} and \eqref{def:p,q-2} define a $\Su$-module structure on $M = \KK[h]^n$ extending the action of $\KK[h]$ by multiplication if and only if \eqref{eq:condition-p,q} holds.

Now, suppose that $M$ has a central character $\chi_M:\KK[z_u]\to\KK$, so that $zm=\chi_M(z)m$, for all $z\in\Z(\Su)=\KK[z_u]$ and all $m\in M$. Set $C = \chi_M(z_u)$. Then we have $xy 1_i = (z_u + u)1_i = (u+C)1_i$, which becomes
\begin{equation}\label{u-cent-char-1}
 (u(h) + C)I = P(h+1)Q(h),
\end{equation}
in matrix form. Then \eqref{eq:condition-p,q} implies that
\begin{equation}\label{u-cent-char-2}
(u(h-1) + C)I = Q(h-1)P(h),
\end{equation}
which translates to $yx 1_i =(u(h-1)+C)1_i$.
Conversely, notice that \eqref{u-cent-char-1} and \eqref{u-cent-char-2} imply \eqref{eq:condition-p,q} and moreover that $M$ has a central character $\chi_M$ with $\chi_M(z_u)=C$.

\begin{exam}
    Let $C \in \KK$, $\lambda \in \Roots_{u+C}$ and $X \subseteq \Roots_{u+C}\setminus\{\lambda\}$. Let $p = \sum_{j=0}^n \alpha_j t^j\in \KK[t]$ of degree $n\ge 1$ and $\E = \E(p,C,\lambda,X)$. By Theorem~\ref{lem:class-h-free-mod}, $\E$ is $\KK[h]$-free with basis $\{e^p,\dotsc,t^{n-1}e^p\}$. Hence there is an isomorphism of $\KK[h]$-modules $\KK[h]^n \rightarrow \E$ such that
    \[
        1_i \mapsto t^{i-1}e^p,\quad i \in [n].
    \]
 Via this isomorphism, $\KK[h]^n$ inherits from $\E$ a structure of $\Su$-module. 
 
Recall, from Definition~\ref{D:exp-mods} and Lemma~\ref{lem:Su-in-A1}, that $h$ acts on $\E$ as $\theta = t\partial + \lambda + 1=\partial t+ \lambda$, $x$ acts as $\partial Q_X(\theta - 1)=Q_X(\theta)\partial$ and $y$ acts as $t P_X(\theta + 1)=P_X(\theta)t$. Set $v_i=t^{i-1}e^p$, for $i\in [n]$. Then we have, for $i\geq 2$, 
\begin{align*}
x\cdot  v_i&=Q_X(\theta)\partial \, t^{i-1}e^p=Q_X(\theta)(\partial t) t^{i-2}e^p
=Q_X(\theta) (\theta-\lambda)v_{i-1}=Q_X(h) (h-\lambda)\cdot v_{i-1}.
\end{align*}
We conclude that $q_{i,j}=Q_X(h) (h-\lambda)\delta_{i-1,j}$, for all $i,j\in [n]$ with $i\geq 2$, where $\delta_{k,\ell}$ is the Kronecker delta. 

Now we take $i=1$:
\begin{align*}
x\cdot  v_1&=Q_X(\theta)\partial \, e^p= Q_X(\theta) p' e^p = Q_X(\theta)\sum_{j=1}^n j\alpha_j t^{j-1}e^p
= Q_X(h)\cdot\sum_{j=1}^n j\alpha_j v_j.
\end{align*}
We conclude that $q_{1,j}=Q_X(h) j\alpha_j$, for all $j\in [n]$. Therefore, we obtain 
\[
       Q(h) =
            Q_X(h)
            \begin{bmatrix}
                \alpha_1    &2\alpha_2      &3\alpha_3  &\dotsb &(n-1)\alpha_{n-1}  &n\alpha_n \\
                (h-\lambda) &0              &0          &\dotsb &0                  &0\\                
                0           &(h-\lambda)    &0          &\dotsb &0                  &0\\
                \vdots      &\vdots         &\vdots     &\ddots &\vdots             &\vdots\\
                0           &0              &0          &\dotsb &0                  &0\\
                0           &0              &\dotsb     &\dotsb &(h-\lambda)        &0
            \end{bmatrix}.
    \]
Similarly, we obtain
    \begin{align*}
        P(h) &=
            P_X(h)
            \begin{bmatrix}
                0    &1    &0  &\dotsb &0  &0\\
                0    &0    &1  &\dotsb &0  &0\\
                0    &0    &0  &\dotsb &0  &0\\
                \vdots    &\vdots    &\vdots  &\ddots &\vdots  &\vdots\\
                0    &0    &0  &\dotsb &0  &1\\
                \frac{(h - (\lambda + 1))}{n\alpha_n}    &-\frac{\alpha_1}{n\alpha_n}    &-\frac{2\alpha_2}{n\alpha_n}  &\dotsb &-\frac{(n-2)\alpha_{n-2}}{n\alpha_n}  &-\frac{(n-1)\alpha_{n-1}}{n\alpha_n}
            \end{bmatrix}\\
            &= P_X(h) \seq{\mathrm{Comp}\left(\frac{tp' - (h-(\lambda+1))}{n\alpha_n}\right)}^{t},
    \end{align*}
    where $\mathrm{Comp}(f(t))$ denotes the companion matrix of $f(t) \in (\KK[h])[t]$, as a polynomial in $t$.
\end{exam}

\subsection{The category $\mathfrak{U}_1$}

Now we will focus on the category $\mathfrak{U}_1$ of $\Su$-modules which are free of rank 1 over $\KK[h]$. In the following, we will identify $M\in \mathfrak{U}_1$ with $\KK[h]$, the (left) regular $\KK[h]$-module. We set $p_M = y\cdot 1$ and $q_M = x\cdot 1$. Whenever there is no ambiguity, we will simply denote these elements of $\KK[h]$ by $p$ and $q$, respectively.

Notice that, by Lemma~\ref{lem:Su-in-A1}, the Casimir element $z_u$ acts on any exponential module by a scalar. Next, we show that this property holds for all modules in $\mathfrak{U}_1$.

\begin{lemma}\label{lem:elements-in-U1-has-central-character}
    Let $M \in \mathfrak{U}_1$. Then $M$ admits a central character $\chi_M$ that satisfies $\chi_M(z_u) = p(h+1)q(h)-u(h) \in \KK$.
\end{lemma}
\begin{proof}
    By equation \eqref{eq:condition-p,q}, we must have
    \[
        q(h-1)p(h) - p(h+1)q(h) = g(h) = u(h-1)-u(h).
    \]
    Let $f(h) = q(h)p(h+1)$. Then we have $f(h-1) - f(h) = g(h)$. By \cite[Lemma~4]{jN15}, the solution of such an equation is unique up to the constant term. Therefore, $f(h) = u(h) + C$, for some $C \in \KK$. In particular $C = p(h+1)q(h)-u(h)$ and
    \[
        z_u \cdot 1 = (xy - u(h))\cdot 1 = x\cdot p - u(h) \stackrel{\eqref{eq:comutation-with-x-and-y}}{=} p(h+1)q(h) - u(h) = C,
    \]
thus proving the lemma.
\end{proof}

Let $M \in \mathfrak{U}_1$ and $C = \chi_M(z_u)$. Let $\xi_C \in \KK^\times$ be the leading coefficient of $u(h)+C$.
Since $u(h) + C = p(h+1)q(h)$, it follows that there is a multiset partition $\Roots_{u+C} =  X\coprod Y$, where $X=\Roots_q$ and $Y=\Roots_{p(h+1)}=\Roots_p-1$. Hence,
\begin{equation*}
q(h) = \xi_q \Poly_X = \xi_q \prod_{\alpha \in X}(h-\alpha) \quad \text{and}\quad p(h) = \xi_p \Poly_{Y+1} = \xi_p \prod_{\alpha \in Y}(h-(\alpha+1)),
\end{equation*}
with $\xi_q,\xi_p \in \KK^\times$ the leading coefficients of $q$ and $p$, respectively, so that $\xi_q\xi_p = \xi_C$. 
Thus, $M=\KK[h]$ is described by $C$, $X$ and $\xi_q$, and we will denote it by $A_C(X,\xi_q)$.

Given $\lambda \in \KK^\times$, let  $\varphi_\lambda$ be the algebra automorphism of $\Su$ defined by
\[
    \varphi_\lambda(x) = \lambda x,\quad \varphi_\lambda(y) = \lambda^{-1}y \quad \text{and}\quad \varphi_\lambda(h) = h.
\]
For any $M \in \Su\md$, define $\mathsf{F}_\lambda M\in \Su\md$ to be the module $M$ with $\Su$-action twisted by $\varphi_\lambda$, i.e, $s\cdot m = \varphi_\lambda(s)m$, for all $s \in \Su$, $m \in \mathsf{F}_\lambda M$. This defines a family of functors
\[
\mathsf{F}_\lambda:\Su\md\longrightarrow\Su\md,
\] 
for all $\lambda \in \KK^\times$.
It is easy to see that $\mathsf{F}_\lambda\mathsf{F}_\mu = \mathsf{F}_{\lambda\mu}$, for all $\lambda, \mu \in \KK^\times$. In particular, the $\mathsf{F}_\lambda$ define category autoequivalences. 

Notice now that $\mathsf{F}_\lambda A_C(X,\xi_q) = A_C(X,\lambda\xi_q)$, for all $\lambda \in \KK^\times$, so in particular $A_C(X,\xi_q)=\mathsf{F}_{\xi_q} A_C(X,1)$. Thus, it suffices to study the modules of the form $A_C(X,1)$, which we simply denote by $A_C(X)$. 

We summarize the above construction.
\begin{defn}\label{D:ACX}
Let $C\in \KK$ and let $X$ be an arbitrary submultiset of the multiset $\Roots_{u+C}$ of roots of $u(h) + C$. Let $Y=\Roots_{u+C}\setminus X$, the multiset complement of $X$ in $\Roots_{u+C}$. Let $q(h) = \Poly_X = \prod_{\alpha \in X}(h-\alpha)$ and $p(h)=\frac{u(h-1) + C}{q(h-1)}\in\KK[h]$.
Then $A_C(X)=\KK[h]$ is the regular $\KK[h]$-module, with action extended to $\Su$ by
\begin{equation*}
x f(h) = f(h+1)q(h)\quad  \text{and}\quad    y f(h) = f(h-1) p(h),\quad \text{for all $f(h)\in\KK[h]$.}
\end{equation*}
\end{defn}

We have proved the following lemma.

\begin{lemma}\label{lem:almost-skeleton-of-U1}
    Let $M \in\Su\md$. Then $M \in \mathfrak{U}_1$ if and only if then there exist $\lambda \in \KK^\times$, $C\in\KK$ and a submultiset $X$ of $\Roots_{u+C}$ such that $M \simeq \mathsf{F}_\lambda A_C(X)$.
\end{lemma}

\begin{lemma}\label{lem:distinct-Ac's}
Let $C, C' \in \KK$, $X$ and $X'$ be submultisets of $\Roots_{u+C}$ and $\Roots_{u+C'}$, respectively, and $\lambda, \lambda'\in\KK^\times$. Then $\mathsf{F}_\lambda A_C(X) \simeq \mathsf{F}_{\lambda'} A_{C'}(X')$ if and only if $C=C'$, $\lambda = \lambda'$ and $X = X'$.
\end{lemma}
\begin{proof}
Assume that $M=\mathsf{F}_\lambda A_C(X) \simeq \mathsf{F}_{\lambda'} A_{C'}(X')=M'$. Then the central characters must be the same, so $C=C'$. Moreover, any isomorphism of $\Su$-modules is in particular an isomorphism of $\KK[h]$-modules, and hence given by multiplication by a nonzero scalar. Thus it can be assumed that the identity map is an isomorphism between the given $\Su$-modules. Then, by checking the action of $x$, we deduce that the isomorphism maps $\lambda q_M$ to $\lambda' q_{M'}$. Hence these polynomials have the same roots and the same leading coefficient, and it follows that $X = X'$ and $\lambda = \lambda'$.
\end{proof}

From Lemmas~\ref{lem:almost-skeleton-of-U1} and \ref{lem:distinct-Ac's} we obtain a classification of the objects in $\mathfrak{U}_1$. 

\begin{cor}\label{cor:skeleton-U1}
    The following family is a skeleton of the category $\mathfrak{U}_1$:
    \[
        \pb{\mathsf{F}_\lambda A_C(X)\mid C\in\KK,\, \lambda\in \KK^\times\text{ and } X \subseteq \Roots_{u+C} \text{ (a submultiset)}}.
    \]
\end{cor}

\subsection{The exponential modules in $\mathfrak{U}_1$}\label{SS:sub:exp-rk-1}

From Theorem~\ref{lem:class-h-free-mod} we know that the exponential modules in $\mathfrak{U}_1$ are precisely those of the form $\E(p,C,\lambda,X)$ with $\degg p = 1$. We will see that the latter exhaust all isomorphism classes in $\mathfrak{U}_1$, except for the isomorphism classes of $A_C(\Roots_{u+C},\alpha)$, with $C,\alpha\in\KK$ and $\alpha\neq 0$. Using the symmetry of the Weyl algebra $A_1$, we define \textit{dual} exponential modules $\E(p,C,\lambda,X)^\vee$ which will cover the remaining isomorphism classes in $\mathfrak{U}_1$.

Fix $p(t)=\alpha t+\beta$, with $\alpha, \beta\in\KK$ and $\alpha\neq 0$. We know that $\E(p,C,\lambda,X)\simeq A_C(\widetilde X, \xi)$, for some submultiset $\widetilde X\subseteq\Roots_{n+C}$ and $\xi\in\KK^{\times}$. These are determined by $x\cdot 1=\xi \Poly_{\widetilde X}(h)$ in $A_C(\widetilde X, \xi)$. Since $\mathrm{End}_{\Su}(A_C(\widetilde X, \xi))=\KK\, 1$, where $1$ stands for the identity on $A_C(\widetilde X, \xi)$, we can assume that the isomorphism $A_C(\widetilde X, \xi) \rightarrow \E(p,C,\lambda,X)$ takes the $\KK[h]$-generators $1\in A_C(\widetilde X, \xi)$ to $e^p\in\E(p,C,\lambda,X)$. Then from $\xi \Poly_{\widetilde X}(h)=x\cdot 1$ we obtain
\begin{align*}
\xi \Poly_{\widetilde X}(h)\cdot e^p= x\cdot e^p=Q_{X}(\theta)\partial e^p=\alpha \Poly_{X}(h)\cdot e^p
\end{align*}
As $\E(p,C,\lambda,X)$ is a free $\KK[h]$-module on $\pb{e^p}$, it follows that $\xi \Poly_{\widetilde X}(h)=\alpha \Poly_{X}(h)$, so $\xi=\alpha$ and $\widetilde X=X$.

Combining the preceding considerations with Lemma~\ref{lem:distinct-Ac's}, we obtain a characterization of the exponential modules for $\Su$ of rank $1$.

\begin{lemma}
Let $p, \widetilde p\in\KK[h]$ with $\degg p=1=\degg \widetilde p$, $C, \widetilde C \in \KK$, $\lambda \in \Roots_{u+C}$, $\widetilde\lambda \in \Roots_{u+\widetilde C}$ and $X \subseteq \Roots_{u+C}\setminus\{\lambda\}$, $\widetilde X \subseteq \Roots_{u+\widetilde C}\setminus\{\widetilde\lambda\}$ submultisets.  The following hold:
\begin{enumerate}[label=\textup{(\alph*)}]
\item $\E(p,C,\lambda,X)\simeq A_C(X, \alpha)$, where $p'(t)=\alpha\in\KK^{\times}$;
\item $\E(p,C,\lambda,X)\simeq\E(\widetilde p,\widetilde C,\widetilde \lambda,\widetilde X)$ if and only if $p'=\widetilde p'$, $C=\widetilde C$ and $X=\widetilde X$.
\end{enumerate}
\end{lemma}

In particular, all modules $A_C(X, \alpha)$ are exponential modules, except for $X=\Roots_{u+C}$. In order to be able to include these latter ones, we define the modules $\E(p,C,\lambda,X)^\vee$ using the symmetry of $A_1$. 

Concretely, let $\function{\tau}{A_1}{A_1}$ be the automorphism defined by $t \mapsto \partial$ and $\partial \mapsto -t$. Then the algebra morphism $\function{\widetilde\varphi_{C,\lambda,X}}{S_u}{A_1}$ defined by $\widetilde\varphi_{C,\lambda,X}=\tau\circ \varphi_{C,\lambda,X}$ induces an $\Su$-module structure on the $A_1$-module $\KK[t]e^p$, denoted by $\E(p,C,\lambda,X)^\vee$. So $h$ acts as $\widetilde\theta =\tau(\theta)= -\partial t + \lambda + 1$, $x$ acts as $-Q_X(\widetilde\theta)t$ and $y$ acts as $P_X(\widetilde\theta)\partial$.

As before, we obtain a characterization of the modules $\E(p,C,\lambda,X)^\vee$ in case $\degg p=1$.
\begin{lemma}
Let $p\in\KK[h]$ with $\degg p=1$, $C \in \KK$, $\lambda \in \Roots_{u+C}$ and $X \subseteq \Roots_{u+C}\setminus\{\lambda\}$ a submultiset.  Then 
\[
\E(p,C,\lambda,X)^\vee\simeq A_C(X\cup\pb{\lambda}, \alpha^{-1}),\quad \text{where $\alpha=p'(t)\in\KK^{\times}$.}
\]
In particular, $A_C(\Roots_{u+C}, \alpha^{-1})\simeq\E(p,C,\lambda,\Roots_{u+C}\setminus\pb{\lambda})^\vee$.
\end{lemma}

\subsection{The submodule structure of $A_C(X)$}\label{SS:sub:structure}

Next, we study the simplicity of the modules $A_C(X)$ and, moreover, we will produce an algorithm to describe the composition series for these modules. We will find that $A_C(X)$ always has finite length as an $\Su$-module.

Unless otherwise noted, throughout this subsection, $C\in\KK$, $X$ denotes an arbitrary submultiset of $\Roots_{u+C}$ and the polynomials $p, q\in\KK[h]$ are as in Definition~\ref{D:ACX}. In particular, $q$ is monic, $X=\Roots_q$ and $Y=\Roots_{u+C}\setminus X = \Roots_{p(h+1)}$. Define
\begin{equation*}
\mathcal L_C(X)=\pb{t(h)\in\KK[h]\mid \text{$t(h)$ is monic, } t(h)\mid t(h-1)p(h) \text{ and } t(h)\mid t(h+1)q(h)}\bigcup \pb{0}.
\end{equation*}
We think of $\mathcal L_C(X)$ as a poset, under the polynomial divisibility relation.

\begin{lemma}\label{lem:L_C(X)-lattice-submod-A_C(X)}
There is an order reversing bijection between $\mathcal L_C(X)$ and the lattice of submodules of $A_C(X)$. Under this correspondence, $t\in\mathcal L_C(X)$ is mapped to $t\KK[h]\subseteq A_C(X)$. Moreover, if $t\neq 0$ then $t\KK[h]\simeq A_C(\Roots_{\overline q})$, where $\overline q=\frac{t(h+1)q(h)}{t(h)}$.
\end{lemma}
\begin{proof}
Let $M$ be an $\Su$-submodule of $A_C(X)$. Then $M$ is an ideal of $\KK[h]$, by restriction, so it follows that $M=t(h)\KK[h]$, for some $t(h)\in\KK[h]$. If $t\neq 0$, then we can assume that $t$ is monic, in which case $M$ determines $t$. Since $M$ is stable under the action of $x$, we have
\begin{equation*}
 t(h+1)q(h)= x t(h)\in M=t(h)\KK[h],
\end{equation*}
thus $t(h)$ divides $t(h+1)q(h)$. Similarly, looking at the action of $y$, we deduce that $t(h)$ divides $t(h-1)p(h)$.

Conversely, let $0\neq t\in\mathcal L_C(X)$ and set $\overline q=\frac{t(h+1)q(h)}{t(h)},\ \overline p=\frac{t(h-1)p(h)}{t(h)}\in\KK[h]$. Notice that $\overline q(h)\overline p(h+1)=q(h)p(h+1)=u(h)+C$, so that $\Roots_{\overline q}$ and $\Roots_{\overline p(h+1)}$ define a partition of $\Roots_{u+C}$. What's more,
\begin{align*}
x t(h)f(h) &= t(h+1)f(h+1)q(h)=t(h)f(h+1)\overline q\quad  \text{and}\\   
y t(h)f(h) &= t(h-1)f(h-1) p(h)=t(h)f(h-1)\overline p,
\end{align*}
for all $f(h)\in\KK[h]$.
Hence $t\KK[h]$ is a submodule of $A_C(X)$ isomorphic to $A_C(\Roots_{\overline q})$. The order reversing property is clear.
\end{proof}

As all nonzero submodules of $A_C(X)$ are of the form $A_C(X')$, for some submultiset $X'$ of $\Roots_{u+C}$, in order to find simplicity criteria and composition series for $A_C(X)$, it suffices to determine all of the maximal submodules of $A_C(X)$. By the previous result, this is tantamount to finding all minimal elements of $\mathcal L_C(X)\setminus\pb{1}$, i.e.\ all $t\in\mathcal L_C(X)$ with no proper nontrivial factors in $\mathcal L_C(X)$.

It will be convenient to introduce a partial order relation on $\KK$, given by $\alpha\preccurlyeq\beta \iff  \beta-\alpha=n 1_\KK$, for some $n\in\N$ (recall that $\chara(\KK)=0$, so $\preccurlyeq$ is indeed antisymmetric).

Let $t\in\mathcal L_C(X)$ and assume that $t\neq 0,1$. Then $\Roots_t$ is a nonempty finite multiset with cardinality equal to $\degg (t)\geq 1$, and thus it  decomposes as a finite union of maximal chains (the connected components of the Hasse diagram of the poset $\Roots_t$):
\begin{align*}
s_1 \ &: \ \alpha_1^1 \preccurlyeq \cdots \preccurlyeq \alpha_{k_1}^1;\\
&\ \vdots\\
s_\ell \ &: \ \alpha_1^\ell \preccurlyeq \cdots \preccurlyeq \alpha_{k_\ell}^\ell.
\end{align*}
Set 
\begin{align*}
t_{s_i}(h)&=\prod_{j=1}^{k_i}(h-\alpha^i_j),  & \overline{t_{s_i}}(h)&=t(h)\seq{t_{s_i}(h)}^{-1},
\end{align*}
so that $t(h)=t_{s_1}(h)\cdots t_{s_\ell}(h)=t_{s_i}(h)\overline{t_{s_i}}(h)$. Thus $t_{s_i}(h)\overline{t_{s_i}}(h)\mid t_{s_i}(h-1)\overline{t_{s_i}}(h-1)p(h)$ and, since $\gcd\seq{t_{s_i}(h), \overline{t_{s_i}}(h)}=1$, the latter is equivalent to 
\begin{align}\label{E:sub:structure:relp}
t_{s_i}(h)\mid t_{s_i}(h-1)\overline{t_{s_i}}(h-1)p(h) \q \text{and} \q \overline{t_{s_i}}(h)\mid t_{s_i}(h-1)\overline{t_{s_i}}(h-1)p(h).
\end{align}
Moreover, as $\gcd\seq{t_{s_i}(h), \overline{t_{s_i}}(h-1)}=1=\gcd\seq{t_{s_i}(h-1), \overline{t_{s_i}}(h)}$, \eqref{E:sub:structure:relp} above is equivalent to 
\begin{align*}
t_{s_i}(h)\mid t_{s_i}(h-1)p(h) \q \text{and} \q \overline{t_{s_i}}(h)\mid \overline{t_{s_i}}(h-1)p(h).
\end{align*}
Replacing $p(h)$ with $q(h)$ we deduce that $t\in\mathcal L_C(X)\iff t_{s_1}, \ldots, t_{s_\ell}\in\mathcal L_C(X)$, and $t\KK[h]=\bigcap_{i=1}^\ell t_{s_i}\KK[h]$. Thus, we may assume that the roots of $t(h)$ form a chain
\begin{align*}
\alpha_1 \preccurlyeq \cdots \preccurlyeq \alpha_{k},
\end{align*}
with $k\geq 1$. Then, from $t(h)\mid t(h-1)p(h)$ we deduce that $\alpha_1$ is a root of $p$, i.e.\ $\alpha_1-1\in \Roots_{p(h+1)}=Y$. Similarly, from $t(h)\mid t(h+1)q(h)$ we deduce that $\alpha_k\in \Roots_q=X$. We call such a multiset a \textit{$(p,q)$-chain}. Note that $\alpha_k-(\alpha_1-1)=\alpha_k-\alpha_1+1\in\ZZ_{\geq 1}$.

\begin{lemma}\label{L:sub:structure:pq}
If $\alpha_1 \preccurlyeq \cdots \preccurlyeq \alpha_{k}$ is a $(p,q)$-chain, then $s(h)=\prod_{j=0}^m (h-(\alpha_1+j))\in\mathcal L_C(X)$, where $m=\alpha_k-\alpha_1$.
\end{lemma}
\begin{proof}
We have 
\begin{align*}
(h-\alpha_1)s(h-1)&=\prod_{j=0}^{m+1}  (h-(\alpha_1+j))=s(h)(h-(\alpha_1+m+1)),\\
(h-(\alpha_1+m))s(h+1)&=\prod_{j=-1}^{m}  (h-(\alpha_1+j))=s(h)(h-(\alpha_1-1)).\\
\end{align*}
Since $h-\alpha_1\mid p(h)$ and $h-\alpha_k=h-(\alpha_1+m)\mid q(h)$, it follows that 
\begin{align*}
s(h)\mid s(h-1)p(h) \q\text{ and }\q s(h)\mid s(h+1)q(h).
\end{align*}
\end{proof}

\begin{cor}\label{cor:simplicity-A_C(X)}
Let $C\in\KK$ and $X$ be a submultiset of $\Roots_{u+C}$, with  $\Roots_{u+C} =  X\coprod Y$. Consider the $\Su$-module $A_C(X)\in \mathfrak{U}_1$, as given in Definition~\ref{D:ACX}. Then $A_C(X)$ is simple if and only if $(X-Y)\cap\ZZ_{\geq 1}=\emptyset$, i.e.\ if and only if there are no $\alpha\in Y$ and $\beta\in X$ such that $\beta-\alpha\in\ZZ_{\geq 1}$.
\end{cor}
\begin{proof}
Assume that $A_C(X)$ is simple and assume, by contradiction, that there exist $\alpha\in Y$ and $\beta\in X$ such that $\beta-\alpha=m+1$, for some $m\in\N$. Let $s(h)=\prod_{j=0}^m (h-(\alpha_1+j))$, with $\alpha_1=\alpha+1$. Then $\degg(s)=m+1\geq 1$ and $\alpha_1 \preccurlyeq \alpha_1+1 \preccurlyeq\cdots \preccurlyeq \alpha_{1}+m=\beta$ is a $(p,q)$-chain, so $s(h)\in\mathcal L_C(X)$ by Lemma~\ref{L:sub:structure:pq}. Hence, $s(h)\KK[h]$ is a proper nontrivial submodule of $A_C(X)$, a contradiction.

Conversely, suppose that $(X-Y)\cap\ZZ_{\geq 1}=\emptyset$ and let $S\subseteq A_C(X)$ be a submodule. Then $S=s(h)\KK[h]$ for some $s(h)\in\mathcal L_C(X)$. If $\degg(s)\geq 1$, then the multiset $\Roots_s$ is finite and nonempty. Thus, by the preceding considerations, there is a divisor $t(h)$ of $s(h)$ with $\degg(t)\geq 1$ such that $\Roots_t$ is a $(p,q)$-chain, say $\alpha_1 \preccurlyeq \cdots \preccurlyeq \alpha_{k}$, with $\alpha_1-1\in Y$ and $\alpha_k\in X$. Thence, $\alpha_k-(\alpha_1-1)=\alpha_k-\alpha_1+1\in\ZZ_{\geq 1}\cap (X-Y)=\emptyset$, a contradiction. Thus, either $s(h)=1$ or $s(h)=0$, proving that $A_C(X)$ is simple.
\end{proof}

Let us return to the classification of the minimal elements $t\in\mathcal L_C(X)\setminus\pb{1}$. We already know that $t=0$ is minimal if and only if $(X-Y)\cap\ZZ_{\geq 1}=\emptyset$, so let's assume that $\degg(t)\geq 1$. From our previous considerations, we know that $\Roots_t$ is a $(p,q)$-chain, say $\alpha_1 \preccurlyeq \cdots \preccurlyeq \alpha_{k}$, with $k\geq 1$.

Suppose there is $1\leq i<k$ such that $\alpha_{i+1}-\alpha_i\geq 2$. Then, writing $t(h)=t_1(h)t_2(h)$ with $t_1(h)=\prod_{j\leq i}(h-\alpha_j)$ and $t_2(h)=\prod_{j> i}(h-\alpha_j)$, the argument we have used before in~\eqref{E:sub:structure:relp} also shows that $t\in\mathcal L_C(X)\iff t_{1}, t_2 \in\mathcal L_C(X)$, and $t\KK[h]=t_{1}\KK[h]\cap t_{2}\KK[h]$. Thus, we may further assume that 
$\alpha_{i+1}-\alpha_i\in\pb{0,1}$, for all $1\leq i<k$. We call such a chain a \textit{gapless chain}.
What's more, taking $m=\alpha_k-\alpha_1$ and $s(h)=\prod_{j=0}^m (h-(\alpha_1+j))$, we see that $s(h)$ divides $t(h)$ (because the chain is gapless) and $s(h)\in\mathcal L_C(X)$, by Lemma~\ref{L:sub:structure:pq}. As $\degg(s)=m+1\geq 1$, it follows from the minimality of $t$ that $t=s$; whence, $t$ is separable. In other words, the roots of $t$ are all distinct and form a gapless $(p,q)$-chain $\alpha_1 \undernegpreccurlyeq \cdots \undernegpreccurlyeq \alpha_{k}$.

\begin{prop}\label{P:sub:structure:min}
Let $C\in\KK$, $X$, $Y$ and $\mathcal L_C(X)$ be as above. Then $t$ is minimal in $\mathcal L_C(X)\setminus\pb{1}$ if and only if either one of the following conditions hold:
\begin{enumerate}[label=(\alph*)]
\item $t=0$ and $(X-Y)\cap\ZZ_{\geq 1}=\emptyset$ (i.e., $A_C(X)$ is simple);\label{P:sub:structure:min:a}
\item $\Roots_t$ is a finite gapless $(p,q)$-chain with no repeated elements, say $\alpha_1 \undernegpreccurlyeq \cdots \undernegpreccurlyeq \alpha_{k}$, with $\degg(t)=k$, such that:\label{P:sub:structure:min:b}
\begin{enumerate}[label=(\roman*)]
\item there is no $i<k$ with $\alpha_i\in X$;\label{P:sub:structure:min:b:i}
\item there is no $i>1$ with $\alpha_i-1\in Y$.\label{P:sub:structure:min:b:ii}
\end{enumerate}
\end{enumerate}
\end{prop}
\begin{proof}
The direct implication is clear from the preceding discussion. Conversely, condition~\ref{P:sub:structure:min:a} clearly implies the minimality of $t=0$. So assume that the roots of $t$ form the chain $\alpha_1 \undernegpreccurlyeq \cdots \undernegpreccurlyeq \alpha_{k}$, satisfying the conditions in~\ref{P:sub:structure:min:b}. In particular, $t\neq0$, as $\Roots_t$ is finite.

Let $s\in\mathcal L_C(X)\setminus\pb{1}$ be a divisor of $t$. As $t\neq 0$, also $s\neq0$ and thus $\degg(s)\geq 1$. Then, $\Roots_s\subseteq \Roots_t$ and the roots of $s$ are of the form $\alpha_{i_1} \undernegpreccurlyeq \cdots \undernegpreccurlyeq \alpha_{i_\ell}$, with $1\leq i_1<\cdots< i_\ell\leq k$. The fact that $s\in\mathcal L_C(X)$ implies that $\Roots_s$ is a $(p,q)$-chain and~\ref{P:sub:structure:min:b} forces  $i_1=1$ and $i_\ell=k$. If $s\neq t$, then there is some $j<\ell$ such that $\alpha_{i_{j+1}}-\alpha_{i_{j}}\geq 2$. Hence, by the argument preceding Proposition~\ref{P:sub:structure:min}, $(h-\alpha_{i_1})\cdots (h-\alpha_{i_j})\in\mathcal L_C(X)$. In particular, $\alpha_{i_j}\in X$, forcing $i_j=k=i_\ell$, which contradicts $j<\ell$. The contradiction implies that $s=t$, proving the minimality of $t$.
\end{proof}

Let $C\in\KK$, $X$, $p,q \in \KK[h]$ and $\mathcal L_C(X)$ be as above. Write $X_0 = X$, $q_0 = q$ and $p_0 = p$. The previous proposition provides a method to construct any decreasing chain of submodules
\begin{equation}\label{eq:decreasing-chain}
    A_C(X_0) \supset t_1\KK[h] \simeq A_C(X_1) \supset t_2t_1\KK[h] \simeq A_C(X_2) \supset \dotsb,
\end{equation}
where $t_i$ is a minimal element of $\mathcal L_C(X_{i-1})\setminus\{1\}$. As long as $t_i\neq 0$, we can proceed with $q_i = \frac{t_i(h+1)q_{i-1}(h)}{t_i(h)}$, $p_i = \frac{t_i(h-1)p_{i-1}(h)}{t_i(h)}$ and $X_i = \Roots_{q_i}$, for $i \ge 1$. The minimality of $t_i$ implies that $A_C(X_{i-1})/A_C(X_{i})$ is simple, for all $i \ge 1$. Then, to prove that all objects in $\mathfrak{U}_1$ have finite length, it is enough to show that, after a finite number of steps, the minimal element obtained is $t_\ell=0$.

\subsection{Composition series for $A_C(X)$}\label{SSS:sub:sub:composition series}

Recall the order $\preccurlyeq$ defined on $\KK$. Given a multiset $Z \subseteq \KK$ and $\beta \in \KK$, denote by $Z_{\undernegpreccurlyeq\beta}$ the submultiset $\{\alpha \in Z \mid \alpha \undernegpreccurlyeq \beta\}$.
Let $X \subseteq \Roots_{u+C}$ be a submultiset and take $\beta \in X$. If $(\Roots_{u+C}\setminus X)_{\undernegpreccurlyeq \beta} \ne \emptyset$, we denote by $X\star \beta$ the submultiset of $\Roots_{u+C}$ defined by
\[
    X \star \beta = \{\hat \beta\} \cup X\setminus \{\beta\},
\]
where $\hat \beta \in (\Roots_{u+C}\setminus X)_{\undernegpreccurlyeq \beta}$ is uniquely defined by imposing the minimum distance from $\beta$, i.e., $\beta - \hat\beta = \min\{\beta - \alpha \mid \alpha \in (\Roots_{u+C}\setminus X)_{\undernegpreccurlyeq \beta}\}$.

Let $t \in \mathcal{L}_C(X)\setminus\pb{1}$ be minimal and suppose that $t\neq 0$. Set $q = \Poly_X$. Then $\Roots_t$ is a finite gapless $(p,q)$-chain with no repeated elements, say $\alpha_1 \undernegpreccurlyeq \cdots \undernegpreccurlyeq \alpha_{k}$ satisfying Proposition~\ref{P:sub:structure:min}\ref{P:sub:structure:min:b}. Set $\overline q = \frac{t(h+1)q(h)}{t(h)}$. It follows that $\Roots_{\overline q} = (X \setminus \{\alpha_k\}) \cup \{\alpha_1-1\}$. Furthermore, by Proposition~\ref{P:sub:structure:min}\ref{P:sub:structure:min:b}\ref{P:sub:structure:min:b:ii},
\[
    k  = \alpha_k - (\alpha_1 - 1 ) = \min\{\alpha_k - \beta \mid \beta \in (\Roots_{u+C}\setminus X)_{\undernegpreccurlyeq \alpha_k} \}.
\]
Thus, $\Roots_{\overline q} = X \star \alpha_k$. Moreover, with this notation, the chain of submodules~\eqref{eq:decreasing-chain} can be written as
\begin{equation}\label{eq:decreasing-chain-star}
    A_C(X) \supset t_1\KK[h] \simeq A_C(X\star \beta_1) \supset t_2t_1\KK[h] \simeq A_C(X\star \beta_1 \star \beta_2) \supset \dotsb,
\end{equation}
where $\beta_i$ is the maximal element of the gapless $(p_{i-1},q_{i-1})$-chain corresponding to $t_i$, a minimal element of $\mathcal{L}_C(X\star \beta_1 \star \dotsb \star \beta_{i-1})\setminus\{1\}$ which we are assuming to be nonzero.

Now, for any submultiset $Z \subseteq \Roots_{u+C}$, define
\[
    \ell(Z) = \sum_{\beta \in Z} |(\Roots_{u+C}\setminus Z)_{\undernegpreccurlyeq\beta}| \ge 0,
\]
where $|\cdot|$ denotes the number of elements of a multiset. Notice that $\ell(Z\star \beta) \leq \ell(Z) -1$, whenever $Z\star \beta$ is defined. Finally, recall also that, by Corollary~\ref{cor:simplicity-A_C(X)}, $A_C(Z)$ is simple if and only if $\ell(Z) = 0$.
Therefore, the chain \eqref{eq:decreasing-chain-star} has maximal length bounded above by $\ell(X)$ and $\ell(X\star \beta_1 \star \dotsb \star \beta_m) = 0$, for some $m\leq\ell(X)$. The last nonzero term of the chain~\eqref{eq:decreasing-chain-star} will be the simple submodule $A_C(X\star \beta_1 \star \dotsb \star \beta_m)$.

From the discussion above we obtain our desired result. 

\begin{prop}\label{prop:finite-length}
    Let $A_C(X) \in \mathfrak{U}_1$. Then $A_C(X)$ has finite length, bounded above by $\ell(X)+1$. 
\end{prop}

The method described above using Proposition~\ref{P:sub:structure:min} and the iterative construction in~\eqref{eq:decreasing-chain-star} gives all possible composition series for $A_C(X)$. Nevertheless, we will see that, regardless of the choices made, the final multiset $X\star \beta_1 \star \dotsb \star \beta_m$ obtained, with $\ell(X\star \beta_1 \star \dotsb \star \beta_m) = 0$, will always be the same. We give an algebraic proof of this result, using the notion of socle of a module $M$, denoted by $\mathrm{soc}(M)$, this being the sum of its simple submodules, or equivalently,
its unique maximal semisimple submodule.

\begin{cor}
Let $A_C(X) \in \mathfrak{U}_1$. Then $\mathrm{soc}\seq{A_C(X)}=A_C(X^\star)$, where $X^\star=X\star \beta_1 \star \dotsb \star \beta_m$ is obtained iteratively by the method described above, terminating with $\ell(X^\star)=0$. In particular, $X^\star$ depends only on $X$.
\end{cor}

\begin{proof}
We have seen that there exist $0\leq m\leq \ell(X)$ and $\beta_1, \ldots, \beta_m\in\Roots_{u+C}$ such that $A_C(X\star \beta_1 \star \dotsb \star \beta_m)$ is a submodule of $A_C(X)$ with $\ell(X\star \beta_1 \star \dotsb \star \beta_m) = 0$, hence simple and thence contained in $\mathrm{soc}\seq{A_C(X)}$.  

The $\KK[h]$-module $A_C(X)$ is just the regular module $\KK[h]$, which contains no nontrivial direct sums of submodules. It follows that the same must hold for $A_C(X)$ as an $\Su$-module. Thus, its socle, being nonzero and semisimple, must be simple and equal to $A_C(X\star \beta_1 \star \dotsb \star \beta_m)$. So $A_C(X\star \beta_1 \star \dotsb \star \beta_m)$ is the unique simple submodule of $A_C(X)$, and the last nonzero term in all composition series for $A_C(X)$. Now, by Lemma~\ref{lem:distinct-Ac's}, the uniqueness of the multiset $X\star \beta_1 \star \dotsb \star \beta_m$ follows.
\end{proof}

\begin{remark}
Let $\Roots_{u+C}=R_1\coprod \cdots\coprod R_k$ be the decomposition of $\Roots_{u+C}$ in to its maximal chains with respect to $\preccurlyeq$. Then $X^\star$ is the unique submultiset of $\Roots_{u+C}$ with $\ell(X^\star)=0$ and $|R_i\cap X^\star|=|R_i\cap X|$, for all $i\in [k]$.
\end{remark}

Next, we will describe the remaining composition factors of $A_C(X)$ and their multiplicities, obtaining as a corollary an exact formula for the length of $A_C(X)$. Since $A_C(X)/\mathrm{soc}\seq{A_C(X)}$ is finite dimensional, $A_C(X^\star)$ occurs with multiplicity one and all the other composition factors, if any, will be finite dimensional.

We summarize the classification of simple $\Su$-modules of finite dimension given by Smith in \cite{spS90} (see also \cite{LR22ca}). Let $\lambda \in \KK$, and $\KK_\lambda = \KK v_\lambda$ be the one-dimensional $\KK[h]$-module where $h$ acts by $\lambda$. Let $\mathfrak{b}\subseteq \Su$ be the unital subalgebra generated by $h$ and $y$. Then $\KK_\lambda$ becomes a $\mathfrak{b}$-module by defining $y v_\lambda = 0$. The Verma module of highest weight $\lambda$ for $\Su$ is defined by
\[
    V(\lambda) = \Su \otimes_{\mathfrak{b}}\KK_\lambda\simeq\KK[x].
\]

\begin{thm}\label{T:Smith:class}\cite{spS90}
    Let $\lambda \in \KK$, then $V(\lambda)$ has a unique maximal submodule and hence a unique simple subquotient, denoted by $L(\lambda)$. Furthermore, any simple $\Su$-module of dimension $j$ is isomorphic to
    \[
        L(\lambda) = V(\lambda)/x^jV(\lambda),
    \]
    for some $\lambda \in \KK$, where $j$ is the minimal positive integer such that $u(\lambda) - u(\lambda - j) = 0$.
\end{thm}

\begin{lemma}\label{lem:quotient-comp-chain}
    Let $A_C(X) \in \mathfrak{U}_1$ and assume that $A_C(X)$ is not simple. Let $0\neq t \in \mathcal{L}_C(X)\setminus\{1\}$ be minimal. Then $A_C(X)/t\KK[h] \simeq L(\beta)$, where $\beta$ is the maximal element of the gapless $(p,q)$-chain corresponding to $t$. 
\end{lemma}
\begin{proof}
    Let $t(h) = (h - (\hat \beta + 1))(h - (\hat \beta +2)) \dotsb (h - \beta)$, with $\hat \beta \undernegpreccurlyeq\beta$. Then $N = A_C(X)/t\KK[h]$ has dimension equal to $\beta - \hat \beta$.
    Define $w\neq 0$ to be the class of $t(h)/(h-\beta)$ in $N$.
    A straightforward computation using the fact that $h-(\hat\beta +1)$ divides $p$ shows that
    \[
        yw = 0,\quad hx^k w = (\beta-k)x^kw, \quad \text{and} \quad yx^{k+1}w = (u(\beta-(k+1))-u(\beta))x^{k} w,
    \]
    for all $k \in \N$. Let $k_0$ be the minimal positive integer such that $x^{k_0}w = 0$. Then we have
    \[
        0 = yx^{k_0}w = (u(\beta-k_0)-u(\beta))x^{k_0-1} w.
    \]
    As $N$ is simple, it follows that $N=\spann_\KK\pb{x^k w\mid k=0, \ldots, k_0-1}$, a simple $\Su$-module of highest weight $\beta$. Thus, $N\simeq L(\beta)$.
    Since $x^{k_0-1} w \ne 0$, this implies that $u(\beta-k_0)-u(\beta) = 0$, and we have $\beta - \hat \beta=\dim_\KK N=k_0$. 
\end{proof}

\begin{remark}
As a converse to the previous result, any simple finite-dimensional $\Su$-module $L(\lambda)$ can be seen as a quotient of $A_C(X) \in \mathfrak{U}_1$, for some $C\in\KK$ and some $X \subseteq \Roots_{u+C}$. Indeed, suppose that $\dim_\KK L(\lambda)=j\geq 1$ and set $C=-u(\lambda)$. Then, by Theorem~\ref{T:Smith:class}, $\lambda, \lambda-j\in \Roots_{u+C}$ and $A_C(\pb{\lambda})$ is well defined. Moreover, $\lambda-j+1\preccurlyeq \cdots \preccurlyeq\lambda$ forms a $(p,q)$-chain for $X=\pb{\lambda}$, so $t_\lambda(h)=\prod_{i=0}^{j-1}(h-(\lambda-i))\in\mathcal L_C(\pb{\lambda})$, by Lemma~\ref{L:sub:structure:pq}. Finally, the minimality of $j$ given in Theorem~\ref{T:Smith:class} ensures, by Proposition~\ref{P:sub:structure:min}, that $t_\lambda\KK[h]$ is a maximal submodule of $A_C(\pb{\lambda})$, isomorphic to $A_C(\pb{\lambda-j})$, and $A_C(\pb{\lambda})/t_\lambda\KK[h] \simeq L(\lambda)$, by Lemma~\ref{lem:quotient-comp-chain}.
\end{remark}

Now, for every submultiset $Z$ of $\Roots_{u+C}$, define the map $\function{\varphi_Z}{\underline{\Roots_{u+C}}}{\N}$ by
\[
    \varphi_Z(\beta) = \min\left\{ |(\Roots_{u+C}\setminus Z)_{\undernegpreccurlyeq\beta}|, |Z_{\succcurlyeq \beta}| \right\},
\]
where $Z_{\succcurlyeq \beta} = \{\alpha \in Z \mid \beta \preccurlyeq \alpha\}$ and $\underline{\Roots_{u+C}}$ is the underlying set obtained from $\Roots_{u+C}$.

We are ready to describe the composition factors of $A_C(X)$ and their multiplicities. It turns out that this is best phrased using the Grothendieck group $K_0(\Su)$, which is the free abelian group on the isomorphism classes of finitely generated $\Su$-modules, modulo the short exact sequences.

\begin{thm}
    Consider the Grothendieck group $K_0(\Su) = \{[M] \mid M \in \Su\md\}$. Let $A_C(X) \in \mathfrak{U}_1$. Then
    \[
        [A_C(X)] = [A_C(X^\star)]+ \sum_{\beta \in \underline{\Roots_{u+C}}} \varphi_X(\beta)[L(\beta)] \in K_0(\Su),
    \]
    where $A_C(X^\star)=\mathrm{soc}\seq{A_C(X)}$.
\end{thm}
\begin{proof}
    We prove it by induction on $\ell (X)$.

    If $\ell (X) = 0$, then clearly $\varphi_X$ is the constant null map, $X=X^\star$ and the claim is proved.
    Suppose that $\ell (X)> 0$ and assume the claim to be true for any submultiset $Z \subseteq \Roots_{u+C}$ such that $\ell(Z)< \ell(X)$. There exists a minimal $t \in \mathcal{L}_C(X)\setminus\{1\}$ and $t\neq 0$, since $\ell (X)> 0$. Let $\beta$ be the maximal element of the gapless $(p,q)$-chain corresponding to $t$ (see Prop.~\ref{P:sub:structure:min}). Then, by Lemmas~\ref{lem:L_C(X)-lattice-submod-A_C(X)} and \ref{lem:quotient-comp-chain}, we have an exact sequence
    \[
        0 \rightarrow A_C(X\star \beta) \rightarrow A_C(X) \rightarrow L(\beta) \rightarrow 0.
    \]
    Then $[A_C(X)] = [A_C(X\star \beta)] + [L(\beta)] \in K_0(\Su)$. Since $\ell(X\star \beta) \leq \ell(X) - 1$, it follows by the induction hypothesis that 
    \[
        [A_C(X)] = [L(\beta)] + [A_C(X^\star)]+ \sum_{\alpha \in \underline{\Roots_{u+C}}} \varphi_{X\star \beta}(\alpha)[L(\alpha)]  \in K_0(\Su),
    \]
    where $A_C(X^\star)=\mathrm{soc}\seq{A_C(X\star\beta)}=\mathrm{soc}\seq{A_C(X)}$.
    So it is sufficient to prove that 
    \begin{equation}\label{eq:phi_x-phi_xstarbeta}
        \varphi_X(\beta) = \varphi_{X\star \beta}(\beta) + 1 \quad \text{and} \quad \varphi_X(\alpha) = \varphi_{X\star \beta}(\alpha), \quad \text{for all $\alpha \ne \beta$}.
    \end{equation}

    Computing $\varphi_X$ and $\varphi_{X\star \beta}$, we obtain:
    \begin{align*}
        |(\Roots_{u+C}\setminus X\star \beta)_{\undernegpreccurlyeq\alpha}| &= 
            \begin{cases}
                |(\Roots_{u+C}\setminus X)_{\undernegpreccurlyeq\alpha}|,&\alpha \ne \beta;\\
                |(\Roots_{u+C}\setminus X)_{\undernegpreccurlyeq\beta}|-1,& \alpha = \beta;
            \end{cases}\\
        |(X\star \beta)_{\succcurlyeq\alpha}| &= 
            \begin{cases}
                |X_{\succcurlyeq\alpha}|,&\alpha \ne \beta;\\
                |X_{\succcurlyeq\beta}|-1,& \alpha = \beta.
            \end{cases}
    \end{align*}
    Thus, \eqref{eq:phi_x-phi_xstarbeta} is satisfied and the theorem is proved.
\end{proof}

\begin{cor}
The module $A_C(X) \in \mathfrak{U}_1$ has length $\displaystyle 1+\sum_{\beta \in \underline{\Roots_{u+C}}} \varphi_X(\beta)$.
\end{cor}


\end{document}